\documentclass[12pt,letterpaper,final,twoside,leqno]{amsart}
\usepackage{amsmath,amssymb,amsthm,amsfonts,amscd,amsopn}
\usepackage{hyperref}
\usepackage[usenames]{color}
\usepackage{eucal, mathrsfs}
\usepackage{amsmath,amssymb,amsthm,%
amsfonts,
amscd,amsopn}
\usepackage[OT1,T1]{fontenc}
\usepackage{hyperref}
\usepackage{times}
\usepackage{xspace}
\usepackage{%epsfig,epic,eepic,latexsym,
color}
\usepackage[all]{xy}\xyoption{dvips}
\usepackage{comment} 
\usepackage{paralist}
\usepackage{enumitem}
\usepackage{supertabular}
\newcounter{are-there-sections}
\newcommand{\ujj}[1]{{%
}}

\def\me{S\'ANDOR J KOV\'ACS\xspace}

\def\mythanks{Supported in part by NSF Grants %s DMS-0554697 and 
 DMS-0856185,  DMS-1301888, and the
Craig McKibben and Sarah Merner Endowed Professorship in Mathematics at the
University of Washington.}

\def\myaddress{University of Washington, Department of Mathematics, Box 354350,
Seattle, WA 98195-4350, USA}

\def\myemail{skovacs@uw.edu\xspace}

\def\myurladdr{http://www.math.washington.edu/$\sim$kovacs\xspace}

\DeclareMathAlphabet{\smallchanc}{OT1}{pzc}%
                                 {m}{it}
\DeclareFontFamily{OT1}{pzc}{}
\DeclareFontShape{OT1}{pzc}{m}{it}%
             {<-> s * [1.100] pzcmi7t}{}
\DeclareMathAlphabet{\mathchanc}{OT1}{pzc}%
                                 {m}{it}
\newcommand{\mcR}{\mathchanc{R}}

\newcommand{\mcb}{\mathchanc{b}}

\newcommand{\mcd}{\mathchanc{d}}

\DeclareFontFamily{OMS}{rsfs}{\skewchar\font'60}
\DeclareFontShape{OMS}{rsfs}{m}{n}{<-5>rsfs5 <5-7>rsfs7 <7->rsfs10 }{}
\DeclareSymbolFont{rsfs}{OMS}{rsfs}{m}{n}
\DeclareSymbolFontAlphabet{\scr}{rsfs}

\newcommand{\sI}{\scr{I}}

\newcommand{\sO}{\scr{O}}

\newcommand{\bC}{\mathbb{C}}

\newcommand{\bN}{\mathbb{N}}

\newcommand{\into}{\hookrightarrow}

\newcommand{\wt}{\widetilde}

\DeclareMathOperator{\codim}{codim}
\DeclareMathOperator{\coker}{{coker}}

\DeclareMathOperator{\exc}{Exc}

\newcommand{\factor}[2]{\left. \raise 2pt\hbox{\ensuremath{#1}} \right/
        \hskip -2pt\raise -2pt\hbox{\ensuremath{#2}}}
\newcommand{\myR}{{\mcR\!}}

\newcommand{\bdot}{\raisebox{.05em}{\hskip-.2em\ensuremath\kdot}}

\newcommand{\kdot}{{{\,\begin{picture}(1,1)(-1,-2)\circle*{2}\end{picture}\ }}}

\newcommand{\Om}{\underline{\Omega}}
\newcommand{\Ox}[1]{\underline{\Omega}_X^{\,#1}}
\newcommand{\Os}[1]{\underline{\Omega}_{\Sigma}^{\,#1}}

\newcommand{\Oxs}[1]{\underline{\Omega}_{X,\Sigma}^{\,#1}}
\newcommand{\Oy}[1]{\underline{\Omega}_Y^{\,#1}}
\newcommand{\Og}[1]{\underline{\Omega}_{\Gamma}^{\,#1}}

\newcommand{\Oyg}[1]{\underline{\Omega}_{Y,\Gamma}^{\,#1}}
\newcommand{\db}{\mcd\mcb}
\def\coh#1.#2.#3.{H^{#1}(#2,#3)}
\def\dimcoh#1.#2.#3.{h^{#1}(#2,#3)}
\def\hypcoh#1.#2.#3.{\mathbb H_{\vphantom{l}}^{#1}(#2,#3)}
\def\loccoh#1.#2.#3.#4.{H^{#1}_{#2}(#3,#4)}
\def\dimloccoh#1.#2.#3.#4.{h^{#1}_{#2}(#3,#4)}
\def\lochypcoh#1.#2.#3.#4.{\mathbb H^{#1}_{#2}(#3,#4)}
\def\ses#1.#2.#3.{0  \longrightarrow  #1   \longrightarrow 
 #2 \longrightarrow #3 \longrightarrow 0} 
\def\sesshort#1.#2.#3.{0
 \rightarrow #1 \rightarrow #2 \rightarrow #3 \rightarrow 0}
\def\dist#1.#2.#3.{  #1   \longrightarrow 
 #2 \longrightarrow #3 \stackrel{+1}{\longrightarrow} } % \tag{$\bigtriangleup$}}
\def\CDdist#1.#2.#3.{  #1   @>>>  #2  @>>>   #3 @>+1>> }  
\def\shortses#1.#2.#3.{0  \rightarrow  #1   \rightarrow 
 #2  \rightarrow   #3 \rightarrow  0}
\def\shortdist#1.#2.#3.{  #1   \rightarrow 
 #2  \rightarrow   #3 \stackrel{+1}{\rightarrow} }  % \tag{$\bigtriangleup$}}
\def\ddist#1.#2.#3.#4.#5.#6.{\CD
#1 @>>> #2 @>>> #3 @>+1>> \\
@VVV @VVV @VVV \\
#4 @>>> #5 @>>> #6 @>+1>> 
\endCD}
\def\ddistun#1.#2.#3.#4.#5.#6.{\CD
#1 @>>> #2 @>>> #3 @>+1>> \\
@. @VVV @VVV  \\
#4 @>>> #5 @>>> #6 @>+1>> 
\endCD}
\def\Iff#1#2#3{
\hfil\hbox{\hsize =#1
\vtop{\noin #2}
\hskip.5cm 
\lower.5\baselineskip\hbox{$\Leftrightarrow$}\hskip.5cm
\vtop{\noin #3}}\hfil\medskip}
\newcommand{\union}\cup
\newcommand{\intersect}\cap
\newcommand{\Union}\bigcup
\newcommand{\Intersect}\bigcap
\def\myoplus#1.#2.{\underset #1 \to {\overset #2 \to \oplus}}

\def\qis{\,{\simeq}_{\text{qis}}\,}

\begin{document}
\makeatletter
\newenvironment{refmr}{}{}
\renewcommand{\labelenumi}{\hskip .5em(\thethm.\arabic{enumi})}
\renewcommand\thesubsection{\thesection.\Alph{subsection}}
\renewcommand\subsection{
  \renewcommand{\sfdefault}{pag}
  \@startsection{subsection}%
  {2}{0pt}{-\baselineskip}{.2\baselineskip}{\raggedright
    \sffamily\itshape\small
  }}
\renewcommand\section{
  \renewcommand{\sfdefault}{phv}
  \@startsection{section} %
  {1}{0pt}{\baselineskip}{.2\baselineskip}{\centering
    \sffamily
    \scshape
    %\bfseries
}}
%\renewcommand\@cite[2]{{\rm [{#1\ifthenelse{\boolean{@tempswa}}{,\nolinebreak[3] #2}{}}]}}
%%%%%%%%%%
\newcounter{lastyear}\setcounter{lastyear}{\the\year}
\addtocounter{lastyear}{-1}
%%%%%%%%%%%
\newcommand\sideremark[1]{%
\normalmarginpar
\marginpar
[
\hskip .45in
\begin{minipage}{.75in}
\tiny #1
\end{minipage}
]
{
\hskip -.075in
\begin{minipage}{.75in}
\tiny #1
\end{minipage}
}}
\newcommand\rsideremark[1]{
\reversemarginpar
\marginpar
[
\hskip .45in
\begin{minipage}{.75in}
\tiny #1
\end{minipage}
]
{
\hskip -.075in
\begin{minipage}{.75in}
\tiny #1
\end{minipage}
}}
%%%%%%%%%
\newcommand\Index[1]{{#1}\index{#1}}
\newcommand\inddef[1]{\emph{#1}\index{#1}}
\newcommand\noin{\noindent}
\newcommand\hugeskip{\bigskip\bigskip\bigskip}
\newcommand\smc{\sc}
\newcommand\dsize{\displaystyle}
\newcommand\sh{\subheading}
\newcommand\nl{\newline}
%%%%%%%%%%%
%%%%%%%%% bibliography helpers
%%%%%%%%%%%
\newcommand\get{\input /home/kovacs/tex/latex/} %$ 
\newcommand\Get{\Input /home/kovacs/tex/latex/} %$ 
\newcommand\toappear{\rm (to appear)}
\newcommand\mycite[1]{[#1]}
\newcommand\myref[1]{(\ref{#1})}
\newcommand\myli{\hfill\newline\smallskip\noindent{$\bullet$}\quad}
\newcommand\vol[1]{{\bf #1}\ } 
\newcommand\yr[1]{\rm (#1)\ } 
%%%%%%%%%%%
%%%%%%%%% text abbreviations
%%%%%%%%%%%
\newcommand\cf{cf.\ \cite}
\newcommand\mycf{cf.\ \mycite}
\newcommand\te{there exist}
\newcommand\st{such that}
%%%%%%%%%%%
%%%%%%%%%%% Theorem Style: BOZONT
%%%%%%%%%%%
\newcommand\myskip{3pt}
\newtheoremstyle{bozont}{3pt}{3pt}%
     {\itshape}%         Body font
     {}%         Indent amount (empty = no indent, \parindent = para indent)
     {\bfseries}% Thm head font
     {.}%        Punctuation after thm head
     {.5em}%     Space after thm head (\newline = linebreak)
     {\thmname{#1}\thmnumber{ #2}\thmnote{ \rm #3}}%         Thm head spec
%%%%%%%%%%%%%%%%%%%%%%%%%%%%%%
\newtheoremstyle{bozont-sf}{3pt}{3pt}%
     {\itshape}%         Body font
     {}%         Indent amount (empty = no indent, \parindent = para indent)
     {\sffamily}% Thm head font
     {.}%        Punctuation after thm head
     {.5em}%     Space after thm head (\newline = linebreak)
     {\thmname{#1}\thmnumber{ #2}\thmnote{ \rm #3}}%         Thm head spec
%%%%%%%%%%%%%%%%%%%%%%%%%%%%%%
\newtheoremstyle{bozont-sc}{3pt}{3pt}%
     {\itshape}%         Body font
     {}%         Indent amount (empty = no indent, \parindent = para indent)
     {\scshape}% Thm head font
     {.}%        Punctuation after thm head
     {.5em}%     Space after thm head (\newline = linebreak)
     {\thmname{#1}\thmnumber{ #2}\thmnote{ \rm #3}}%         Thm head spec
%%%%%%%%%%%%%%%%%%%%%%%%%%%%%%
\newtheoremstyle{bozont-remark}{3pt}{3pt}%
     {}%         Body font
     {}%         Indent amount (empty = no indent, \parindent = para indent)
     {\scshape}% Thm head font
     {.}%        Punctuation after thm head
     {.5em}%     Space after thm head (\newline = linebreak)
     {\thmname{#1}\thmnumber{ #2}\thmnote{ \rm #3}}%         Thm head spec
%%%%%%%%%%%%%%%%%%%%%%%%%%%%%%
\newtheoremstyle{bozont-def}{3pt}{3pt}%
     {}%         Body font
     {}%         Indent amount (empty = no indent, \parindent = para indent)
     {\bfseries}% Thm head font
     {.}%        Punctuation after thm head
     {.5em}%     Space after thm head (\newline = linebreak)
     {\thmname{#1}\thmnumber{ #2}\thmnote{ \rm #3}}%         Thm head spec
%%%%%%%%%%%%%%%%%%%%%%%%%%%%%%
\newtheoremstyle{bozont-reverse}{3pt}{3pt}%
     {\itshape}%         Body font
     {}%         Indent amount (empty = no indent, \parindent = para indent)
     {\bfseries}% Thm head font
     {.}%        Punctuation after thm head
     {.5em}%     Space after thm head (\newline = linebreak)
     {\thmnumber{#2.}\thmname{ #1}\thmnote{ \rm #3}}%         Thm head spec
%%%%%%%%%%%%%%%%%%%%%%%%%%%%%%
\newtheoremstyle{bozont-reverse-sc}{3pt}{3pt}%
     {\itshape}%         Body font
     {}%         Indent amount (empty = no indent, \parindent = para indent)
     {\scshape}% Thm head font
     {.}%        Punctuation after thm head
     {.5em}%     Space after thm head (\newline = linebreak)
     {\thmnumber{#2.}\thmname{ #1}\thmnote{ \rm #3}}%         Thm head spec
%%%%%%%%%%%%%%%%%%%%%%%%%%%%%%
\newtheoremstyle{bozont-reverse-sf}{3pt}{3pt}%
     {\itshape}%         Body font
     {}%         Indent amount (empty = no indent, \parindent = para indent)
     {\sffamily}% Thm head font
     {.}%        Punctuation after thm head
     {.5em}%     Space after thm head (\newline = linebreak)
     {\thmnumber{#2.}\thmname{ #1}\thmnote{ \rm #3}}%         Thm head spec
%%%%%%%%%%%%%%%%%%%%%%%%%%%%%%
\newtheoremstyle{bozont-remark-reverse}{3pt}{3pt}%
     {}%         Body font
     {}%         Indent amount (empty = no indent, \parindent = para indent)
     {\sc}% Thm head font
     {.}%        Punctuation after thm head
     {.5em}%     Space after thm head (\newline = linebreak)
     {\thmnumber{#2.}\thmname{ #1}\thmnote{ \rm #3}}%         Thm head spec
%%%%%%%%%%%%%%%%%%%%%%%%%%%%%%
\newtheoremstyle{bozont-def-reverse}{3pt}{3pt}%
     {}%         Body font
     {}%         Indent amount (empty = no indent, \parindent = para indent)
     {\bfseries}% Thm head font
     {.}%        Punctuation after thm head
     {.5em}%     Space after thm head (\newline = linebreak)
     {\thmnumber{#2.}\thmname{ #1}\thmnote{ \rm #3}}%         Thm head spec
%%%%%%%%%%%%%%%%%%%%%%%%%%%%%%
\newtheoremstyle{bozont-def-newnum-reverse}{3pt}{3pt}%
     {}%         Body font
     {}%         Indent amount (empty = no indent, \parindent = para indent)
     {\bfseries}% Thm head font
     {}%        Punctuation after thm head
     {.5em}%     Space after thm head (\newline = linebreak)
     {\thmnumber{#2.}\thmname{ #1}\thmnote{ \rm #3}}%         Thm head spec
%%%%%%%%%%%%%%%%%%%%%%%%%%%%%%
%%%%%%%%%%%%%%%%%%%%%%%%%%%%%% thms
%%%%%%%%%%%%%%%%%%%%%%%%%%%%%%
\theoremstyle{bozont}    
\ifnum \value{are-there-sections}=0 {%
  \newtheorem{proclaim}{Theorem}
} 
\else {%
  \newtheorem{proclaim}{Theorem}[section]
} 
\fi
%%%%%%%%%%%%%%%%%%%%%%%%%%%%%%
%%%%%%%%%%%%%%%%%%%%%%%%%%%%%%
\newtheorem{thm}[proclaim]{Theorem}
\newtheorem{mainthm}[proclaim]{Main Theorem}
\newtheorem{cor}[proclaim]{Corollary} 
\newtheorem{cors}[proclaim]{Corollaries} 
\newtheorem{lem}[proclaim]{Lemma} 
\newtheorem{prop}[proclaim]{Proposition} 
\newtheorem{conj}[proclaim]{Conjecture}
\newtheorem{subproclaim}[equation]{Theorem}
\newtheorem{subthm}[equation]{Theorem}
\newtheorem{subcor}[equation]{Corollary} 
\newtheorem{sublem}[equation]{Lemma} 
\newtheorem{subprop}[equation]{Proposition} 
\newtheorem{subconj}[equation]{Conjecture}
%%%%
\theoremstyle{bozont-sc}
\newtheorem{proclaim-special}[proclaim]{\specialthmname}
\newenvironment{proclaimspecial}[1]
     {\def\specialthmname{#1}\begin{proclaim-special}}
     {\end{proclaim-special}}
%%%%%%%%%%%%%%%%%%%%%%%%%%%%%%
\theoremstyle{bozont-remark}
\newtheorem{rem}[proclaim]{Remark}
\newtheorem{subrem}[equation]{Remark}
\newtheorem{notation}[proclaim]{Notation} 
\newtheorem{assume}[proclaim]{Assumptions} 
\newtheorem{obs}[proclaim]{Observation} 
\newtheorem{example}[proclaim]{Example} 
\newtheorem{examples}[proclaim]{Examples} 
\newtheorem{complem}[equation]{Complement}%!!!!!!!!!!!!!!!!!!!!!!
\newtheorem{const}[proclaim]{Construction}   %!!!!!!!!!!!!!!!!
\newtheorem{ex}[proclaim]{Exercise} 
\newtheorem{subnotation}[equation]{Notation} 
\newtheorem{subassume}[equation]{Assumptions} 
\newtheorem{subobs}[equation]{Observation} 
\newtheorem{subexample}[equation]{Example} 
\newtheorem{subex}[equation]{Exercise} 
\newtheorem{claim}[proclaim]{Claim} 
\newtheorem{inclaim}[equation]{Claim} 
\newtheorem{subclaim}[equation]{Claim} 
\newtheorem{case}{Case} 
\newtheorem{subcase}{Subcase}   
\newtheorem{step}{Step}
\newtheorem{approach}{Approach}
\newtheorem{fact}{Fact}
\newtheorem{subsay}{}
%%\newcommand{\Subheading}[1]
%%{\def\SubHeadingName{#1}\begin{SubHeading}\end{SubHeading}}%  
\newtheorem*{SubHeading*}{\SubHeadingName}%
\newtheorem{SubHeading}[proclaim]{\SubHeadingName}
\newtheorem{sSubHeading}[equation]{\sSubHeadingName}
\newenvironment{demo}[1] {\def\SubHeadingName{#1}\begin{SubHeading}}
  {\end{SubHeading}}%  
\newenvironment{subdemo}[1]{\def\sSubHeadingName{#1}\begin{sSubHeading}}
  {\end{sSubHeading}} %
\newenvironment{demo-r}[1]{\def\SubHeadingName{#1}\begin{SubHeading-r}}
  {\end{SubHeading-r}}%
\newenvironment{subdemo-r}[1]{\def\sSubHeadingName{#1}\begin{sSubHeading-r}}
  {\end{sSubHeading-r}} %
\newenvironment{demo*}[1]{\def\SubHeadingName{#1}\begin{SubHeading*}}
  {\end{SubHeading*}}%
\newtheorem{defini}[proclaim]{Definition}
\newtheorem{question}[proclaim]{Question}
\newtheorem{subquestion}[equation]{Question}
\newtheorem{crit}[proclaim]{Criterion}
\newtheorem{pitfall}[proclaim]{Pitfall}%!!!!!!!!!!!!!!!!!!!!!!
\newtheorem{addition}[proclaim]{Addition}%!!!!!!!!!!!!!!!!!!!!!!
\newtheorem{principle}[proclaim]{Principle} %!!!!!!!!!!!!!!!!!!!!!!
%%%%%%%%%%%%%%%%%%%%%%%%%%%%%%
%%%  these were at once \theoremstyle{bozont-def}
\newtheorem{condition}[proclaim]{Condition}
\newtheorem{say}[proclaim]{}
\newtheorem{exmp}[proclaim]{Example}
\newtheorem{hint}[proclaim]{Hint}
\newtheorem{exrc}[proclaim]{Exercise}
\newtheorem{prob}[proclaim]{Problem}
\newtheorem{ques}[proclaim]{Question}    %!!!!!!!!!!!!!!!!!!!!
\newtheorem{alg}[proclaim]{Algorithm}
\newtheorem{remk}[proclaim]{Remark}          
\newtheorem{note}[proclaim]{Note}            
\newtheorem{summ}[proclaim]{Summary}         
\newtheorem{notationk}[proclaim]{Notation}   
\newtheorem{warning}[proclaim]{Warning}  
\newtheorem{defn-thm}[proclaim]{Definition--Theorem}  %!!!!!!!!!!!!!!!!!!!!!!!!
\newtheorem{convention}[proclaim]{Convention}  %!!!!!!!!!!!!!!!!!!!!!!!!!!!
%%%%%%%%%%%%%%%%%%%%%%%%%%%%%%
\newtheorem*{ack}{Acknowledgment}
\newtheorem*{acks}{Acknowledgments}
%%%%%%%%%%%%%%%%%%%%%%%%%%%%%%
\theoremstyle{bozont-def}    
\newtheorem{defn}[proclaim]{Definition}
\newtheorem{subdefn}[equation]{Definition}
%%%%%%%%%%%%%%%%%%
\theoremstyle{bozont-reverse}    
\newtheorem{corr}[proclaim]{Corollary} 
\newtheorem{lemr}[proclaim]{Lemma} 
\newtheorem{propr}[proclaim]{Proposition} 
\newtheorem{conjr}[proclaim]{Conjecture}
%%%%
\theoremstyle{bozont-reverse-sc}
\newtheorem{proclaimr-special}[proclaim]{\specialthmname}
\newenvironment{proclaimspecialr}[1]%
{\def\specialthmname{#1}\begin{proclaimr-special}}%
{\end{proclaimr-special}}
%%%%%%%%%%%%%%%%%%%%%%%%%%%%%%
\theoremstyle{bozont-remark-reverse}
\newtheorem{remr}[proclaim]{Remark}
\newtheorem{subremr}[equation]{Remark}
\newtheorem{notationr}[proclaim]{Notation} 
\newtheorem{assumer}[proclaim]{Assumptions} 
\newtheorem{obsr}[proclaim]{Observation} 
\newtheorem{exampler}[proclaim]{Example} 
\newtheorem{exr}[proclaim]{Exercise} 
\newtheorem{claimr}[proclaim]{Claim} 
\newtheorem{inclaimr}[equation]{Claim} 
\newtheorem{SubHeading-r}[proclaim]{\SubHeadingName}
\newtheorem{sSubHeading-r}[equation]{\sSubHeadingName}
\newtheorem{SubHeadingr}[proclaim]{\SubHeadingName}
\newtheorem{sSubHeadingr}[equation]{\sSubHeadingName}
\newenvironment{demor}[1]{\def\SubHeadingName{#1}\begin{SubHeadingr}}{\end{SubHeadingr}}
\newtheorem{definir}[proclaim]{Definition}
%%%%%%%%%%%%%%%%%%%%%%%%%%%%%%
\theoremstyle{bozont-def-newnum-reverse}    
\newtheorem{newnumr}[proclaim]{}
%%%%%%%%%%%%%%%%%%%%%%%%%%%%%%
\theoremstyle{bozont-def-reverse}    
\newtheorem{defnr}[proclaim]{Definition}
\newtheorem{questionr}[proclaim]{Question}
\newtheorem{newnumspecial}[proclaim]{\specialnewnumname}
\newenvironment{newnum}[1]{\def\specialnewnumname{#1}\begin{newnumspecial}}{\end{newnumspecial}}
%%%%%%%%%%%%%%%%%%
\numberwithin{equation}{proclaim}
\numberwithin{figure}{section} 
\newcommand\equinsect{\numberwithin{equation}{section}}
\newcommand\equinthm{\numberwithin{equation}{proclaim}}
\newcommand\figinthm{\numberwithin{figure}{proclaim}}
\newcommand\figinsect{\numberwithin{figure}{section}}
%% The next environment produces equations that are numbered within the section, not the
%% theorem. It also increases the counter for thm. This is to be used at the
%% beginning of a section to avoid reference numbers containing 0. It may also be
%% used for equations that are not part of a numbered statement. Otherwise equations
%% take the last numbered environment's number and this is not always desirable.
\newenvironment{sequation}{%
\numberwithin{equation}{section}%
\begin{equation}%
}{%
\end{equation}%
\numberwithin{equation}{thm}%
\addtocounter{thm}{1}%
}
%%%%%%%%%%%%%%%
\newcommand{\num}{\arabic{section}.\arabic{proclaim}}
%%%%%%%%%%%%%%%%%
\newenvironment{pf}{\smallskip \noindent {\sc Proof. }}{\qed\smallskip}
%%%%%%%%%%%%%%%%%
\newenvironment{enumerate-p}{
  \begin{enumerate}}
  {\setcounter{equation}{\value{enumi}}\end{enumerate}}
\newenvironment{enumerate-cont}{
  \begin{enumerate}
    {\setcounter{enumi}{\value{equation}}}}
  {\setcounter{equation}{\value{enumi}}
  \end{enumerate}}
%%%%%%%%%%%
\let\lenumi\labelenumi
\newcommand{\rmlabels}{\renewcommand{\labelenumi}{\rm \lenumi}}
\newcommand{\rmlabelsoff}{\renewcommand{\labelenumi}{\lenumi}}
%%%%%%%%%%%
\newenvironment{heading}{\begin{center} \sc}{\end{center}}
%%%%%%%%%%%
\newcommand\subheading[1]{\smallskip\noindent{{\bf #1.}\ }}
%%%%%%%%%%%%%%%%%%%%%%%%%%%%%%
\newlength{\swidth}
\setlength{\swidth}{\textwidth}
\addtolength{\swidth}{-,5\parindent}
\newenvironment{narrow}{
\begin{thm}
  Let $X$ be a normal variety and $\pi: Y\to X$ a resolution of singularities.  Let
  $\Sigma\subseteq X$ be a subvariety and set $E=\exc(\pi)$ and
  $\Gamma=E\cup(\pi^{-1}\Sigma)_\text{red}$.  Assume that $(X,\Sigma)$ is a Du~Bois
  pair.  Then for all $p$,
  \begin{equation*}
    \myR^{\dim X-1}\pi_*\big(\Omega_{Y}^p(\log\Gamma)(-\Gamma)\big) =0.
  \end{equation*}  
\end{thm}

  \medskip\noindent\hfill\begin{minipage}{\swidth}}
  {\end{minipage}\medskip}
%%%%%%%%%%%%%%%%%%%%%%%%%%%%%%
\newcommand\nospace{\hskip-.45ex}
\makeatother
%%

%%%%%%%%%%%%%%%%%%%%%%%%%
\title{Steenbrink vanishing extended}
\author{\me}
\date{\today}
\thanks{\mythanks}
\address{\myaddress}
\email{\myemail}
\urladdr{\myurladdr}
%
%\keywords{}
%
\maketitle
%\tableofcontents
%%%%%%%%%%%%%%%%%%%%%%%%%
%%%%%%%%%%%%%%%%%%%%%%%%%
\newcommand{\szabores}{Szab\'o-resolution\xspace}
\newcommand{\pairD}{\Delta}
%%%%%%%%%%%%%%%%%%%%%%%%%

\centerline{\it To Steven Kleiman on the occasion of his 70$^{\text{th}}$ birthday} %

\begin{abstract}
  The notion of \emph{DB index}, a measure of how far a singularity of a pair is from
  being Du~Bois, is introduced and used to generalize vanishing theorems of
  \cite{Steenbrink85}, \cite{MR2854859}, and \cite{MR2784747} with simpler and more
  natural proofs than the originals. An argument used in one of these proofs also
  yields an additional theorem connecting various push forwards that lie outside of
  the range of the validity of the above vanishing theorems.
\end{abstract}

\section{Introduction}

The importance of rational singularities has been demonstrated for decades through
various applications. Log terminal singularities (of all stripes) are rational and
this single fact has far reaching consequences in the minimal model program.
Unfortunately, not all singularities that appear in the minimal model program are
rational. In particular, the class of log canonical singularities which emerges as
the most important class in many applications, for instance in moduli theory, is not
necessarily rational. 

The class of Du~Bois singularities is an enlargement of the class of rational
singularities. Even though this notion was introduced several decades ago
\cite{Steenbrink83}, it has remained relatively obscure for a long time. It was
recently proved that log canonical singularities are Du~Bois \cite{KK10} and this
fact has started a flurry of activities and Du~Bois singularities are becoming
central in the minimal model program and related areas.

An important application of Du~Bois singularities appeared in \cite{MR2854859} and in
some other articles that grew out of it \cite{Druel-LZ,Graf-LZ}.  The way Du~Bois
singularities were used in these articles is through a vanishing theorem that can be
considered a generalization of a vanishing theorem due to Steenbrink
\cite{Steenbrink85}.

The notion of Du~Bois singularities was recently extended for pairs in
\cite{MR2784747} and the purpose of the present article is to extend the vanishing
theorem used in \cite{MR2854859} to Du~Bois pairs. In particular, the following is
proved.

\begin{thm}\label{thm:main}
  Let $X$ be a normal variety and $\pi: Y\to X$ a resolution of singularities.  Let
  $\Sigma\subseteq X$ be a subvariety and set $E=\exc(\pi)$ and
  $\Gamma=E\cup(\pi^{-1}\Sigma)_\text{red}$.  Assume that $(X,\Sigma)$ is a Du~Bois
  pair and that $\Gamma$ is an snc divisor.  Then for all $p$,
  \begin{equation*}
    \myR^{\dim X-1}\pi_*\big(\Omega_{Y}^p(\log\Gamma)(-\Gamma)\big) =0.
  \end{equation*}  
\end{thm}

In fact, a stronger version will be proved in Corollary~\ref{cor:gkkp-van}, but
stating the stronger form requires some preparation, done in \S\ref{sec:db-index}.
Theorem~\ref{thm:main} is a generalization of \cite[Thm.~14.1]{MR2854859}.  As noted
in \cite{MR2854859} it follows from Steenbrink's vanishing theorem
\cite[Thm.~2(b)]{Steenbrink85} in the cases $p>1$.  In
\S\ref{sec:gener-steenbr-vanish} a generalization,
Theorem~\ref{thm:gener-steenbr-vanish}, of Steenbrink's vanishing is proved after
reviewing the definition and basic properties of the Deligne-Du~Bois complex in
\S\ref{sec:deligne-du-bois}.  Arguably the proof presented here is much simpler than
Steenbrink's original proof and perhaps more importantly, at least in the opinion of
the author, more natural and makes it more clear why the statement is true.  The same
is true for the $p=1$ case of Theorem~\ref{thm:main} which is proved in
\S\ref{sec:vanishing-p=1}. The final section, \S\ref{sec:an-accidental-exact},
contains a theorem that could be considered a byproduct of the proof in
\S\ref{sec:vanishing-p=1}.

\begin{demo}{Definitions and Notation}\label{demo:defs-and-not}
  Unless otherwise stated, all objects are assumed to be defined over $\bC$, all
  schemes are assumed to be of finite type over $\bC$ and a morphism means a morphism
  between schemes of finite type over $\bC$.

  If $\phi:Y\to Z$ is a birational morphism, then $\exc(\phi)$ will denote the
  \emph{exceptional set} of $\phi$. For a closed subscheme $W\subseteq X$, the ideal
  sheaf of $W$ is denoted by $\sI_{W\subseteq X}$ or if no confusion is likely, then
  simply by $\sI_W$.  For a point $x\in X$, $\kappa(x)$ denotes the residue field of
  $\sO_{X,x}$.

  %%%% Let $X$ be a complex scheme (i.e., a scheme of finite type over $\bC$) of
  %%%% dimension n. Let $D_{\rm filt}(X)$ denote the derived category of filtered
  %%%% complexes of $\sO_{X}$-modules with differentials of order $\leq 1$ and
  %%%% $D_{\rm filt, coh}(X)$ the subcategory of $D_{\rm filt}(X)$ of complexes $\cx
  %%%% K$, such that for all $i$, the cohomology sheaves of $Gr^{i}_{\rm filt}\cx K$
  %%%% are coherent cf.\ \cite{DuBois81}, \cite{GNPP88}.  Let $D(X)$ and $D_{\rm
  %%%%   coh}(X)$ denote the derived categories with the same definition except that
  %%%% the complexes are assumed to have the trivial filtration.  The superscripts
  %%%% $+, -, b$ carry the usual meaning (bounded below, bounded above, bounded).
  %%%% Isomorphism in these categories is denoted by $\qis$.  A sheaf $\sF$ is also
  %%%% considered as a complex $\sF^\kdot$ with $\sF^0=\sF$ and $\sF^i=0$ for $i\neq
  %%%% 0$.  If $\cx K$ is a complex in any of the above categories, then $h^i(\cx K)$
  %%%% denotes the $i$-th cohomology sheaf of $\cx K$.

  The right derived functor of an additive functor $F$, if it exists, is denoted by
  $\myR F$ and $\myR^iF$ is short for $h^i\circ \myR F$, where $h^i$ is the
  $i^\text{th}$ cohomology sheaf of a complex. 

  % Furthermore, $\bH^i$, $\bH^i_{\rm c}$,
  % $\bH^i_Z$ , and $\sH^i_Z$ will denote $\myR^i\Gamma$, $\myR^i\Gamma_{\rm c}$,
  % $\myR^i\Gamma_Z$, and $\myR^i\sH_Z$ respectively, where $\Gamma$ is the functor
  % of global sections, $\Gamma_{\rm c}$ is the functor of global sections with
  % proper support, $\Gamma_Z$ is the functor of global sections with support in the
  % closed subset $Z$, and $\sH_Z$ is the functor of the sheaf of local sections with
  % support in the closed subset $Z$.  Note that according to this terminology, if
  % $\phi\col Y\to X$ is a morphism and $\sF$ is a coherent sheaf on $Y$, then
  % $\myR\phi_*\sF$ is the complex whose cohomology sheaves give rise to the usual
  % higher direct images of $\sF$.

  % Finally, we will also make the following simplification in notation. First
  % observe that if $\iota:\Sigma \into X$ is a closed embedding of schemes then
  % $\iota_*$ is exact and hence $\myR\iota_*=\iota_*$. This allows one to make the
  % following harmless abuse of notation: If $\sfA\in\ob D(\Sigma)$, then, as usual
  % for sheaves, we will drop $\iota_*$ from the notation of the object
  % $\iota_*\sfA$. In other words, we will, without further warning, consider $\sfA$
  % an object in $D(X)$.
\end{demo}

\begin{ack}
  The author is grateful to Daniel Greb and Stefan Kebekus for useful discussions
  that helped clarifying the author's thoughts about the topics studied in this
  article and to the referee for the repeated efforts in finding several typos in
  Chapter 6.
\end{ack}

\section{The Deligne-Du~Bois complex}\label{sec:deligne-du-bois}

The {Deligne-Du~Bois} complex is a generalization of the de~Rham complex to singular
varieties.  It is a complex of sheaves on $X$ that is quasi-isomorphic to the
constant sheaf $\bC_X$. The terms of this complex are harder to describe but its
properties, especially cohomological properties are very similar to the de~Rham
complex of smooth varieties. In fact, for a smooth variety the {Deligne-Du~Bois}
complex is quasi-isomorphic to the de~Rham complex, so it is indeed a direct
generalization.

The original construction of this complex, $\Ox\bdot$, is based on simplicial
resolutions. The reader interested in the details is referred to the original article
\cite{DuBois81}.  Note also that a simplified construction was later obtained in
\cite{Carlson85} and \cite{GNPP88} via the general theory of polyhedral and cubic
resolutions.  An easily accessible introduction can be found in \cite{Steenbrink85}.
Other useful references are the recent book \cite{PetersSteenbrinkBook} and the
survey \cite{Kovacs-Schwede09}. We will actually not use these resolutions here. They
are needed for the construction, but if one is willing to believe the listed
properties (which follow in a rather straightforward way from the construction) then
one should be able follow the material presented here.  The interested reader should
note that recently Schwede found a simpler alternative construction of (part of) the
{Deligne-Du~Bois} complex that does not need a simplicial resolution
\cite{MR2339829}.  For applications of the {Deligne-Du~Bois} complex and {Du~Bois}
singularities other than the ones listed here see \cite{Steenbrink83}, \cite[Chapter
12]{Kollar95s}, \cite{Kovacs99,Kovacs00c,KSS10,KK10}.
%%% The word ``hyperresolution'' will refer to either a simplicial, polyhedral, or
%%% cubicresolution. Formally, the construction of $\Ox\bdot$ is the same regardless
%%% the typeof resolution used and no specific aspects of either types will be used.

As mentioned in the introduction, the Deligne-Du~Bois theory was recently extended
for pairs in \cite{MR2784747}. In particular, we will be using the Deligne-Du~Bois
complex of a pair. I will not repeat the construction or all the basic properties of
this complex. The interested reader may consult the original article \cite{MR2784747}
or the more recent and more detailed account in \cite[\S6]{SingBook}.
In particular, the most important properties of the Deligne-Du~Bois complex of a
pair are listed in \cite[Theorem~6.5]{SingBook}. I will only recall the ones that are
used here. 

\begin{defini}
  A \emph{reduced pair} consists of $X$ a reduced scheme of finite type over $\bC$
  and $\Sigma\subseteq X$ a reduced closed subscheme of $X$.  The Deligne-Du~Bois
  complex of the reduced pair $(X,\Sigma)$ is denoted by $\Oxs\bdot$. If
  $\Sigma=\emptyset$, then this reduces to the Deligne-Du~Bois complex of $X$ and
  will be denoted by $\Ox\bdot$.

  The shifted associated graded quotients of $\Oxs\bdot$ are denoted and defined by
  the following formula:
  $$\quad
  \Oxs p:= Gr^p_{\text{filt}}\Ox\bdot[p].
  $$
\end{defini}

\medskip

\begin{demo}{Basic Properties}
  The Deligne-Du~Bois complex is a resolution of the sheaf
  $j_!\bC_{X\setminus\Sigma}$, the constant sheaf $\bC$ on $X\setminus \Sigma$
  extended by $0$ to the entire $X$ \cite[Thm.~4.1]{MR2784747},
  \cite[Thm.~6.5(1)]{SingBook}:
  \begin{equation}
    \label{eq:12}
    j_!\bC_{X\setminus\Sigma}\qis \Oxs\bdot 
  \end{equation}
  Here, of course, $j:X\setminus \Sigma\into X$ is the natural inclusion map.

  If $X$ is smooth and $\Sigma$ is an snc divisor on $X$, then the shifted associated
  graded quotients have only one non-zero cohomology sheaf, $h^0$, and that one is
  given by logarithmic differentials vanishing along $\Sigma$
  \cite[(3.10)]{MR2784747}, \cite[Thm.~6.5(3)]{SingBook}, that is, if $(X,\Sigma)$ is
  an snc pair, then
  \begin{equation}
    \label{eq:4}
    \Oxs p \qis \Omega_X^p(\log \Sigma)(-\Sigma) \simeq \Omega_X^p(\log \Sigma)\otimes
    \sI_{\Sigma\subseteq X}
  \end{equation}

  The following two distinguished triangles will be important in the sequel. 

  The first one connects the Deligne-Du~Bois complex of a pair to that of the members
  of the pair. See \cite{MR2784747} or \cite[Thm.~6.5(7)]{SingBook}: For any $p\in
  \bN$ there exists a distinguished triangle,
  \begin{equation}
    \label{eq:1}
    \xymatrix{%
      & \Oxs p \ar[r] & \Ox p \ar[r] & \Os p \ar[r]^-{+1} &.
    }
  \end{equation}

  The second one relates the Deligne-Du~Bois complexes of pairs connected via a
  birational morphism.  See \cite{DuBois81} or \cite[Thm.~6.5(10)]{SingBook}: Let
  $\pi: Y\to X$ be a projective morphism and $\Sigma\subseteq X$ a reduced closed
  subscheme such that $\pi$ is an isomorphism over $X\setminus\Sigma$. Set
  $\Gamma=(\pi^{-1}\Sigma)_\text{red}$. Then for any $p\in\bN$ there exists a
  distinguished triangle,
  \begin{equation}
    \label{eq:2}
    \xymatrix{ 
      & \Ox p \ar[r] & \Os p \oplus \myR\pi_* \Oy p \ar[r] & \myR\pi_*\Og p
      \ar[r]^-{+1} &.\quad 
    }
  \end{equation}

  One more property of the Deligne-Du~Bois complex will be very useful. This may be
  considered a Grauert-Riemenschneider-type vanishing theorem in this setting. See
  \cite[V.6.2]{GNPP88}: For any $X$ a reduced scheme of finite type over $\bC$,
  \begin{equation}
    \label{eq:3}\quad
    h^q(\Ox p) = 0, \text{ for } p+q>\dim X
  \end{equation}

\end{demo}

\section{Generalized Steenbrink vanishing}\label{sec:gener-steenbr-vanish}

The following theorem extends Steenbrink's vanishing theorem to a more general
situation.

\begin{thm}\label{thm:gener-steenbr-vanish}
  Let $\pi: Y\to X$ be a projective morphism of reduced schemes of finite type over
  $\bC$ and $\Sigma\subseteq X$ a reduced closed subscheme such that $\pi$ is an
  isomorphism over $X\setminus\Sigma$. Assume that $\Sigma$ does not contain any
  irreducible components of $X$.  Set $\Gamma=(\pi^{-1}\Sigma)_\text{red}$.  Then for
  any $x\in X$,
  $$(\myR^q\pi_*\Oyg p)_x=0, \text{ for } p+q > \dim_x X.$$
  In particular,
  $$\myR^q\pi_*\Oyg p =0, \text{ for } p+q > \dim X.$$
\end{thm}

\begin{proof}
  Let $x\in X$ and $n:=\dim_x X$.  The statement is local on $X$, so we may restrict
  to a neighborhood of $x$ and assume that $\dim X=\dim_x X$.

  Consider the long exact sequence of cohomology sheaves induced by the distinguished
  triangle (\ref{eq:2}):
  $$
  \cdots \to h^q(\Ox p) \to h^q(\Os p) \oplus \myR^q\pi_*\Oy p \to \myR^q \pi_*\Og p
  \to h^{q+1}(\Ox p) \to \cdots .
  $$
  By (\ref{eq:3}) the morphism 
  \begin{equation}
    \label{eq:5}
    \myR^q\pi_*\Oy p \to \myR^q \pi_*\Og p
  \end{equation}
  is an isomorphism for $p+q>n$ and a surjection for $p+q=n$. Notice that this is the
  place where we use the assumption that $\Sigma$ does not contain any irreducible
  components of $X$: Indeed that implies that then $\dim_x\Sigma<n$ and hence for
  $p+q=n$, we still have that $h^q(\Os p)=0$. Without this we could only conclude
  that $h^q(\Os p) \oplus \myR^q\pi_*\Oy p \to \myR^q \pi_*\Og p$ is surjective, but
  this is not sufficient in the next step.

  Next consider the long exact sequence induced by applying the functor $\myR\pi_*$
  to the distinguished triangle from (\ref{eq:1}) on $Y$:
  $$
  \cdots\!\to\!\myR^{q-1}\!\pi_*\Oy p \to \!\myR^{q-1}\!\pi_*\Og p \!\to
  \myR^q\pi_*\Oyg p \to \myR^q\pi_*\Oy p \to \myR^q \pi_*\Og p \to\!\cdots.
  $$
  Now the desired statement follows from (\ref{eq:5}).
\end{proof}

This way we obtained a very simple proof of Steenbrink's vanishing theorem:

\begin{thm}\label{thm:steenbr-vanish}
  \cite[Thm.~2(b)]{Steenbrink85} Let $X$ be a complex variety, $\Sigma\subseteq X$
  such that $X\setminus \Sigma$ is smooth and $\pi:Y\to X$ a proper birational
  morphism such that $Y$ is smooth, $E=\pi^{-1}\Sigma$ is an snc divisor on $Y$ and
  $\pi$ induces an isomorphism between $Y\setminus E$ and $X\setminus \Sigma$.  Then
  $$
  \myR^q\pi_* \Omega_Y^p(\log E)(-E) = 0, \text{ for } p+q>\dim X.
  $$
\end{thm}

\begin{proof}
  By (\ref{eq:4}) this is a direct consequence of \eqref{thm:gener-steenbr-vanish}.
\end{proof}

\section{DB index and more vanishing}\label{sec:db-index}

First we will extend the notion of Du~Bois pairs as follows.

\begin{defini}
  Let $(X,\Sigma)$ be a reduced pair and $x\in X$ a point. The \emph{local DB index
    of $(X,\Sigma)$ at $x$}, $\db_x(X,\Sigma)$, is the smallest natural number above
  which the cohomology sheaves of the $0^\text{th}$ shifted associated graded complex
  of the Deligne-Du~Bois complex of $(X,\Sigma)$ vanish at $x$. In other words,
  \begin{equation*}
    \db_x(X,\Sigma)= \min \{ q\in \bN \ \vert \  h^i(\Oxs 0)_x=0 \text{ for } i> q
    \}. 
  \end{equation*}
  Notice that $\db_x(X\Sigma)$ is an upper semicontinuos function of $x$.

  The \emph{DB index of $(X,\Sigma)$} is the maximum of the local DB indices of
  $(X,\Sigma)$ at $x$ for all $x\in X$, that is,
  \begin{equation*}
    \db(X,\Sigma)= \max \{ \db_x(X,\Sigma) \ \vert \ x\in X    \}. 
  \end{equation*}

  Observe that if $(X,\Sigma)$ is a Du~Bois pair, then $\db(X,\Sigma)=0$, but the
  converse is not true. For instance, any reduced curve $C$ has $\db(C,\emptyset)=0$,
  but $C$ is Du~Bois if and only if it is seminormal.
\end{defini}

\begin{claim}(cf.\ \cite[6.7]{SingBook})
  \label{claim:db-index-bound}
  Let $(X,\Sigma)$ be a reduced pair and $x\in X$ a point.  If $\Sigma$ does not
  contain any irreducible component of $X$, then
  $$ \db_x(X,\Sigma)\leq \dim_x X-1. $$
\end{claim}

\begin{proof}
  Consider the long exact sequence of cohomology sheaves induced by the distinguished
  triangle (\ref{eq:1}):
  \begin{equation}
    \label{eq:9}
    h^{q-1}(\Os 0)\to h ^q(\Oxs 0) \to h^q(\Ox 0)
  \end{equation}

  The assumption implies that $\dim_x\Sigma < \dim_x X$ and then the fact (cf.\
  \cite[6.6]{SingBook}) that $h^q(\Ox 0)=0$ for all $q\geq \dim X$ imply that the
  left and right hand side of (\ref{eq:9}) are zero for $q\geq \dim_x X$ and hence so
  is the middle.
\end{proof}

Using the DB index we generalize the vanishing theorem \cite[6.1]{MR2784747}.

\begin{thm}\label{thm:du-bois-pairs}
  Let $\pi: Y\to X$ be a projective birational morphism of reduced schemes of finite
  type over $\bC$ and $\Sigma\subseteq X$ a reduced closed subscheme such that
  $\Sigma$ does not contain any irreducible components of $X$.  Set $E=\exc(\pi)$,
  $\Gamma=E\cup(\pi^{-1}\Sigma)_\text{red}$, and $Z=\overline{\pi(E)\setminus
    \Sigma}\subseteq X$.  Then for any $x\in X$,
  $$
  (\myR^q\pi_*\Oyg 0)_x=0, \text{ for } q > \max \{ \db_x(X,\Sigma), \db_x(Z,Z\cap
  \Sigma) +1 \}.
  $$ 
  In particular,
  %setting $s:=\min\{ s(x) \ \vert \ x\in X \}$ we have that
  $$
  \myR^q\pi_*\Oyg 0 =0, \text{ for } q > \max \{ \db(X,\Sigma), \db(Z,Z\cap \Sigma)
  +1 \}.
  $$
\end{thm}

\begin{proof}
  Let $x\in X$.  The statement is local on $X$ and $\db_x$ is upper semicontinuous as
  a function of $x$, so we may restrict to a neighborhood of $x$ and assume that
  $\db(X,\Sigma)=\db_x(X,\Sigma)$ and $\db(Z,Z\cap\Sigma)=\db_x(Z,Z\cap \Sigma)$.

  Using the long exact sequence of cohomology sheaves associated to the distinguished
  triangle of (\ref{eq:1}) shows that the morphism
  \begin{equation}
    \label{eq:6}
    \alpha^q: h^q(\Ox 0)\to h^q(\Os 0)
  \end{equation}
  is an isomorphism for $q>\db(X,\Sigma)$. 
  and a surjection for $q=\db(X,\Sigma)$.

  Next let ${\wt\Sigma}=Z\cup \Sigma$. Then \cite[3.19]{MR2784747} implies that for
  any $x\in X$
  \begin{equation*}
    \db_x({\wt\Sigma},\Sigma)=\db_x(Z,Z\cap\Sigma), 
  \end{equation*}
  and hence that $h^q(\Om_{{\wt\Sigma},\Sigma}^0)=0$ for $q>\db(Z,Z\cap\Sigma)$. The
  same way as above this implies that the morphism
  \begin{equation}
    \label{eq:7}
    \beta^q: h^q(\Om_{\wt\Sigma}^0)\to h^q(\Os 0)
  \end{equation}
  is an isomorphism for $q>\db(Z,Z\cap\Sigma)$.
  % and a surjection for $q=\db(Z,Z\cap\Sigma)$.

  Now observe that the natural morphism $\Ox 0 \to \Os 0$ factors through
  $\Om_{\wt\Sigma}^0$ and hence $\alpha^q=\beta^q\circ \gamma^q$ where
  $\gamma^q:h^q(\Ox 0) \to h^q(\Om_{\wt\Sigma}^0)$ is the natural morphism
  corresponding to the embedding ${\wt\Sigma}\subseteq X$. It follows from
  (\ref{eq:6}) and (\ref{eq:7}) that
  
  \begin{multline}
    \label{eq:8} 
    \text{ $\gamma^q$ is an isomorphism for $q>\max \{ \db(X,\Sigma), \db(Z,Z\cap
      \Sigma)\}$,}\\ \qquad\text{ and a surjection for $q>\max \{ \db(X,\Sigma)-1,
      \db(Z,Z\cap \Sigma)\}$.}
  \end{multline}

  Next consider the long exact sequence of cohomology sheaves induced by the
  distinguished triangle of (\ref{eq:2}) for the pairs $(X,{\wt\Sigma})$ and
  $(Y,\Gamma)$:
  $$
  \cdots \to h^q(\Ox 0) \to h^q(\Om_{\wt\Sigma}^ 0) \oplus \myR^q\pi_*\Oy 0 \to
  \myR^q \pi_*\Og 0 \to h^{q+1}(\Ox 0) \to \cdots .
  $$

  By (\ref{eq:8}) it follows that the natural morphism
  $$
  \myR^q\pi_*\Oy 0 \to \myR^q \pi_*\Og 0
  $$
  is 
  \begin{center}
    an isomorphism for $q>\max \{ \db(X,\Sigma), \db(Z,Z\cap \Sigma)\}$, and \\
    a surjection for $q>\max \{ \db(X,\Sigma)-1, \db(Z,Z\cap \Sigma)\}$\phantom{,
      and}
  \end{center}
  As before, using the long exact sequence induced by (\ref{eq:1}) we obtain that
  $$
  \myR^q\pi_*\Oyg 0 = 0 \text{ for } q>\max \{ \db(X,\Sigma), \db(Z,Z\cap
  \Sigma)+1\}. 
  $$
\end{proof}

This theorem has several interesting consequences.

\begin{cor}\label{cor:dimZ}
  Let $X$ be an irreducible variety and $\pi: Y\to X$ a projective birational
  morphism.  Let $\Sigma\subseteq X$ be a subvariety and set $E=\exc(\pi)$,
  $Z=\overline{\pi(E)\setminus \Sigma}$, and
  $\Gamma=E\cup(\pi^{-1}\Sigma)_\text{red}$. 
  Then
  $$
  \myR^q\pi_*\Oyg 0 =0, \text{ for } q > \max( \db(X,\Sigma), \dim Z).
  $$
\end{cor}

\begin{proof}
  By the definition of $Z$, $\Sigma$ cannot contain any irreducible component of $Z$
  and hence \eqref{claim:db-index-bound} implies that $\db(Z,Z\cap\Sigma)\leq \dim
  Z-1$. Therefore the statement follows from \eqref{thm:du-bois-pairs}.
\end{proof}

Finally we obtain a Steenbrink-type theorem for $p=0$ with some assumption on the
singularities of the pair $(X,\Sigma)$.

\begin{cor}
  Let $X$ be an irreducible variety and $\pi: Y\to X$ a resolution of singularities.
  Let $\Sigma\subseteq X$ be a subvariety and set $E=\exc(\pi)$,
  $Z=\overline{\pi(E)\setminus \Sigma}$, and
  $\Gamma=E\cup(\pi^{-1}\Sigma)_\text{red}$.  Assume that $\db(X,\Sigma)\leq \dim Z$
  and that $\Gamma$ is an snc divisor.  Then
  \begin{equation}
    \label{eq:10}
    \myR^q\pi_*\sI_{\Gamma\subseteq Y} =0, \text{ for } q > \dim Z.
  \end{equation}
  In particular, if $X$ is normal of dimension $n\geq 2$, then
  \begin{equation}
    \label{eq:11}
    \myR^{n-1}\pi_*\sI_{\Gamma\subseteq Y} =0.
  \end{equation}
\end{cor}

\begin{proof}
  Follows from (\ref{eq:4}) and \eqref{cor:dimZ}.
\end{proof}

\begin{rem}
  Note that for (\ref{eq:11}) one only needs that $\db(X,\Sigma) \leq n-2$. This
  should be considered a mild assumption given that $\db(X,\Sigma) \leq n-1$ always
  holds.
\end{rem}

\section{Vanishing for $p=1$}\label{sec:vanishing-p=1}

Notice that the assumption $p+q>\dim X$ in Steenbrink's theorem
\eqref{thm:steenbr-vanish}, as well as its generalization
\eqref{thm:gener-steenbr-vanish} means that these theorems are vacuous in the cases
$p\leq 1$, since $\myR^q\pi_*=0$ for $q\geq \dim X$ anyway.

We obtained an extension of this theorem under additional conditions for the case
$p=0$ in \ref{thm:du-bois-pairs}.  It turns out that by a simple argument one may
extend this vanishing also for the case $p=1$ and $q=n-1$.

First we need the following:

\begin{lem}[(Topological vanishing)]\cite[Lemma~14.4]{MR2854859}
  \label{lem:topvanishing}%
  
  \noindent
  Let $\pi: Y\to X$ be a projective morphism of reduced schemes of finite type over
  $\bC$ and $\Sigma\subseteq X$ a reduced closed subscheme such that $\pi$ is an
  isomorphism over $X\setminus\Sigma$.  Set $\Gamma=(\pi^{-1}\Sigma)_\text{red}$, $j:
  Y \setminus \Gamma \hookrightarrow Y$ the inclusion map, and $j_!\mathbb{C}_{Y
    \setminus \Gamma}$ the constant sheaf $\bC$ on $Y\setminus \Gamma$ extended by
  $0$ to the entire $Y$.
  Then $\myR^q\pi_*\bigl( j_!\mathbb{C}_{Y \setminus \Gamma}\bigr) = 0$ for all
  $q>0$.
\end{lem}
\begin{proof}
  The proof of \cite[Lemma~14.4]{MR2854859} works verbatim.
\end{proof}

\begin{thm}\label{thm:p=1}
  Let $\pi: Y\to X$ be a projective birational morphism of reduced schemes of finite
  type over $\bC$ and $\Sigma\subseteq X$ a reduced closed subscheme such that
  $\Sigma$ does not contain any irreducible components of $X$.  Set $E=\exc(\pi)$,
  $\Gamma=E\cup(\pi^{-1}\Sigma)_\text{red}$, and $Z=\overline{\pi(E)\setminus
    \Sigma}\subseteq X$.  Assume that $\db_x(X,\Sigma)\leq \dim_x X-2$ for all $x\in
  X$ and that $\codim_XZ\geq 2$ .  Then for any $x\in X$,
  $$(\myR^{\dim_xX-1}\pi_*\Oyg 1)_x=0.$$
  In particular,
  $$\myR^{\dim X-1}\pi_*\Oyg 1 =0.$$
\end{thm}

\begin{subrem}\label{rem:dim-is-at-least-2}
  Notice that the assumption on the DB index implies that $\dim X\geq \dim_x X\geq 2$
  for any $x\in X$.
\end{subrem}

\begin{proof}
  The statement is local on $X$, so we may restrict to a neighborhood of any $x\in X$
  and assume that $\db(X,\Sigma)=\db_x(X,\Sigma)$ and $\dim X=\dim_x X$. Let $n=\dim
  X$, which is at least $2$ by \eqref{rem:dim-is-at-least-2}.

  The quasi-isomorphism of (\ref{eq:12}) and the filtration of $\Oxs\bdot$ induces a
  spectral sequence computing $\myR^{i}\pi_*\bigl( j_!\mathbb{C}_{Y \setminus
    \Gamma}\bigr)$:
  \begin{equation} 
    \label{eq:13} 
    E_1^{p,q}=\myR^q\pi_*\Oxs p \Rightarrow 
    E_\infty^{p,q}= \myR^{p+q}\pi_*\bigl( j_!\mathbb{C}_{Y
      \setminus       \Gamma}\bigr) 
  \end{equation} 
  By \eqref{lem:topvanishing}, $E_\infty^{p,q}=0$ for $p+q>0$, so all $E_1^{p,q}$ in
  that range have to be killed in the spectral sequence.

  Next consider the differentials in the spectral sequence that map to or from
  $E_r^{1,n-1}$ for some $r>0$:
  \begin{align*}
    d_r: & \ E_r^{1-r, n+r-2} \to E_r^{1,n-1}\\
    d_r: & \qquad E_r^{1, n-1} \to E_r^{r+1,n-r}.
  \end{align*}
  Observe that $E_r^{1-r, n+r-2}=0$ trivially if $r>1$ and $E_1^{0,n-1}=0$ by
  \eqref{thm:du-bois-pairs} (cf.\ \ref{cor:dimZ}). Furthermore, $E_r^{r+1,n-r}=0$ by
  \eqref{thm:gener-steenbr-vanish} and hence the only way $E_\infty^{1,n-1}=0$ can
  happen is if already $E_1^{1,n-1}=0$ which is exactly the desired statement.
\end{proof}

This way we obtain the promised generalization and simplified proof of
\cite[Thm.~14.1]{MR2854859}

\begin{cor}\label{cor:gkkp-van}
  Let $X$ be a normal variety and $\pi: Y\to X$ a resolution of singularities.  Let
  $\Sigma\subseteq X$ be a subvariety and set $E=\exc(\pi)$ and
  $\Gamma=E\cup(\pi^{-1}\Sigma)_\text{red}$.  Assume that $\db(X,\Sigma)\leq \dim
  X-2$ and that $\Gamma$ is an snc divisor.  Then for all $p$,
  \begin{equation*}
    \myR^{\dim X-1}\pi_*\big(\Omega_{Y}^p(\log\Gamma)(-\Gamma)\big) =0.
  \end{equation*}
\end{cor}

\begin{proof}
  This is a direct consequence of the combination of \eqref{thm:gener-steenbr-vanish}
  for $p>1$, \eqref{thm:du-bois-pairs} for $p=0$, and \eqref{thm:p=1} for $p=1$.
\end{proof}

\section{An accidental exact sequence}\label{sec:an-accidental-exact}

As a sort of byproduct of the argument used to prove \eqref{thm:p=1} we also obtain
the following.

\begin{thm}
  Let $X$ and $Y$ be irreducible varieties and $\pi: Y\to X$ be a projective
  birational morphism.  Let $\Sigma\subseteq X$ be a subvariety and set
  $E=\exc(\pi)$, $Z=\overline{\pi(E)\setminus \Sigma}$, and
  $\Gamma=E\cup(\pi^{-1}\Sigma)_\text{red}$.  Assume that $\db(X,\Sigma)\leq \dim
  X-3$ and that $\codim_XZ\geq 3$.  Then there exists a 5-term exact sequence,
  $$
  \myR^{n-3}\pi_*\Oyg{2} \to \myR^{n-3}\pi_*\Oyg{3} \to \myR^{n-2}\pi_*\Oyg{1} \to
  \myR^{n-2}\pi_*\Oyg{2} \to 0.
  $$
\end{thm}

\begin{subrem}
  Note that the morphisms $\myR^{n-3}\pi_*\Oyg{2} \to \myR^{n-3}\pi_*\Oyg{3}$ and
  $\myR^{n-2}\pi_*\Oyg{1} \to \myR^{n-2}\pi_*\Oyg{2}$ in the above sequence are
  natural maps induced by the filtration on $\Oyg \bdot$. However, the map
  $\myR^{n-3}\pi_*\Oyg{3}\to \myR^{n-2}\pi_*\Oyg{1}$ is actually the inverse of a
  natural map from a subsheaf of $\myR^{n-2}\pi_*\Oyg{1}$ to a quotient sheaf of
  $\myR^{n-3}\pi_*\Oyg{3}$ that turns out to be an isomorphism.

  Also note that $\myR^{n-2}\pi_*\Oyg{1} \to \myR^{n-2}\pi_*\Oyg{2}$ is already
  surjective under the assumptions of Theorem~\ref{thm:p=1}.
\end{subrem}

\begin{proof}
  We will use the spectral sequence and notation introduced in the proof of
  \eqref{thm:p=1}. In particular, first consider the differentials
  \begin{equation}
    \label{eq:14}
    d_r: E_r^{t, n-t} \to E_r^{r+t,n-t-r+1},
  \end{equation}
  for any $t$, and observe that (as above) $E_r^{r+t,n-t-r+1}=0$ for all $r\geq 1$ by
  \eqref{thm:gener-steenbr-vanish}.  Next consider the differentials
  $$
  d_r: E_r^{2-r, n+r-3} \to E_r^{2,n-2},
  $$
  and observe that (as above) that $E_r^{2-r, n+r-3}=0$ trivially if $r>2$ and
  $E_2^{0,n-1}=0$ by \eqref{thm:du-bois-pairs} (cf.\ \ref{cor:dimZ}). Therefore, the
  only way $E_\infty^{2,n-2}=0$ can happen is if
  \begin{equation}
    \label{eq:19}
      d_1: E_1^{1, n-2} \to E_1^{2,n-2}
   \end{equation}
   is surjective. By \eqref{thm:du-bois-pairs} the differential $d_1:E_1^{0, n-2}\to
   E_1^{1, n-2}$ is $0$ (this is where we need the stronger assumptions on
   $\db(X,\Sigma)$ and $\codim_XZ$), so the kernel of the morphism in (\ref{eq:19})
   is equal to $E_2^{1,n-2}$, i.e., we have an exact sequence
  \begin{equation}
    \label{eq:16}
    \xymatrix{
      E_2^{1,n-2} \ar[r]^{\ker d_1} &  E_1^{1,n-2} \ar[r]^{d_1} & E_1^{2, n-2} \ar[r]
      & 0, 
    }
  \end{equation}
  and again by \eqref{thm:gener-steenbr-vanish} and \eqref{thm:du-bois-pairs} it
  follows that in order for $E_\infty^{1,n-2}=0$ to hold the next differential
  \begin{equation}
    \label{eq:15}
    d_2: E_2^{1,n-2} \into E_2^{3, n-3}
  \end{equation}
  has to be injective. 

  Next, by (\ref{eq:14}) we see that $E_1^{3,n-3}$ has to be killed by differentials
  mapping \emph{to} it, that is, by the differentials
  $$
  d_r: E_r^{3-r, n+r-4} \to E_r^{3,n-3}.
  $$
  As before, $E_r^{3-r, n+r-4}=0$ for $r>3$, and $E_r^{0, n-1}=0$ by
  \eqref{thm:du-bois-pairs}, so there are two differentials, $d_1$ and $d_2$ that can
  kill $E_1^{3,n-3}$. It follows that $E_2^{3,n-3}$ is the cokernel of $ d_1: E_1^{2,
    n-3} \to E_1^{3,n-3}$, i.e., we have an exact sequence
  \begin{equation}
    \label{eq:17}
    \xymatrix{ E_1^{2, n-3} \ar[r]^{d_1} & E_1^{3,n-3} \ar[r]^{{\coker d_1}} & 
      E_2^{3,n-3}, }    
  \end{equation}
  and that $d_2: E_2^{1,n-2} \to E_2^{3, n-3}$ has to be surjective. However, we have
  already seen in (\ref{eq:15}) that this $d_2$ is injective and hence it must be an
  isomorphism:
  \begin{equation}
    \label{eq:18}
    \xymatrix{
      d_2: E_2^{1,n-2} \ar[r]^-{\simeq} &  E_2^{3, n-3}.
    }
  \end{equation}
  Putting together (\ref{eq:16}), (\ref{eq:17}), and (\ref{eq:18}) gives the desired
  exact sequence: 
  $$
  \xymatrix{%
    E_1^{2, n-3} \ar[r]^{d_1} & E_1^{3,n-3} \ar[rrrr]^{{(\ker d_1)\,\circ\,
        (d_2^{-1})\,\circ\, (\coker d_1)}} &&&& E_1^{1,n-2} \ar[r]^{d_1} & E_1^{2,
      n-2} \ar[r] & 0. }
  $$
\end{proof}

\begin{rem}
  It is left for the reader to formulate the consequence of this theorem in the style
  of \eqref{cor:gkkp-van}.
\end{rem}

%\bibliographystyle{/home/kovacs/tex/TeX_input/skalpha} %$ 
%\bibliography{/home/kovacs/tex/TeX_input/Ref} %$

\def\cprime{$'$} \def\polhk#1{\setbox0=\hbox{#1}{\ooalign{\hidewidth
  \lower1.5ex\hbox{`}\hidewidth\crcr\unhbox0}}} \def\cprime{$'$}
  \def\cprime{$'$} \def\cprime{$'$} \def\cprime{$'$}
  \def\polhk#1{\setbox0=\hbox{#1}{\ooalign{\hidewidth
  \lower1.5ex\hbox{`}\hidewidth\crcr\unhbox0}}} \def\cdprime{$''$}
  \def\cprime{$'$} \def\cprime{$'$} \def\cprime{$'$} \def\cprime{$'$}
\providecommand{\bysame}{\leavevmode\hbox to3em{\hrulefill}\thinspace}
\providecommand{\MR}{\relax\ifhmode\unskip\space\fi MR}
\providecommand{\MRhref}[2]{%
  \href{http://www.ams.org/mathscinet-getitem?mr=#1}{#2}
}
\providecommand{\href}[2]{#2}

\end{document}